\documentclass{amsart}
\usepackage{amsmath, amssymb, amsthm, fullpage}
\theoremstyle{plain} 
\newtheorem{theorem}{Theorem}
\newtheorem{proposition}{Proposition}
\newtheorem{lemma}[theorem]{Lemma}
\numberwithin{equation}{section}

\def\CC{\mathbb C}

\def\QQ{\mathbb Q}
\def\RR{\mathbb R}
\def\ZZ{\mathbb Z}
\def\Re{\operatorname{Re}}
\def\Im{\operatorname{Im}}
\def\us{\underset}
\def\conj{\overline} 
\def\mf{\mathfrak}
\renewcommand{\pod}[1]{\allowbreak\mathchoice
  {\if@display \mkern 18mu\else \mkern 8mu\fi (#1)}
  {\if@display \mkern 18mu\else \mkern 8mu\fi (#1)}
  {\mkern4mu(#1)}
  {\mkern4mu(#1)}
}

\begin{document}

\baselineskip=17pt

\title{On Eisenstein primes}
\author{Mayank Pandey}
\maketitle  
\section{Introduction and statement of results}

In this paper, we prove the following result:
\begin{theorem}

$$\underset{\ell^2 + \ell m + m^2 \le x}{\sum\sum} \Lambda(2\ell - m)\Lambda(\ell^2 - \ell m + m^2) \sim \sigma x$$

for some $\sigma > 0$.

\end{theorem}

We shall prove Theorem 1 by following along the lines of the proof of Theorem 20.3 in \cite{FI2}, by using $\QQ(\omega)$ 
rather than $\QQ(i)$ when working with the bilinear forms that arise in Section 20.4 of \cite{FI2}. A related result was proved
by Fouvry and Iwaniec in \cite{FoI} where it is shown that there are infinitely many primes of the form $\ell^2 + m^2$ such that 
$\ell$ is prime.

\section{Preliminaries}
Let $\gamma_\ell = \log \ell$ when $\ell$ is a prime greater than 2 and $0$ otherwise. 
Then, let $$a_n = \sum_{\ell^2 - \ell m + m^2 = n}\gamma_{2\ell - m} = \sum_{r^2 + 3s^2 = 4n}\gamma_{r}.$$

Let $$A(x) = \sum_{n\le x} a_n$$ 
and let $$A_d(x) = \underset{n\equiv 0\pmod{d}}{\sum_{n\le x}} a_n$$

Let $\rho(d) = |\{v\in\ZZ/(d) : v^2 + 3 \equiv 0\pmod{d}\}|$.

We expect that $A_d(x)$ is well approximated by 
$$M_d(x) = \frac{\rho(4d)}{4d} \sum_{r\le\sqrt{4x}}\frac{1}{2}\gamma_r \sqrt{\frac{4x - r^2}{3}}$$
so we let the remainder terms $r_d(x)$ be such that 
$$A_d(x) = M_d(x)+ r_d(x)$$
For $d$ even, this is clearly equal to 0, while for $d$ odd, since $\rho(d)$ is multiplicative, this is equal to 
$$\frac{\rho(d)}{4d} \sum_{r\le\sqrt{4x}}\gamma_r \sqrt{\frac{4x - r^2}{3}}$$

We then have the following:

\begin{proposition}

Suppose that for some $\sqrt{x} < D\le x(\log x)^{-20},$ 
\begin{equation}R(x; D) = \sup_{y\le x} \sum_{d\le D} |r_d(y)|\ll  A(x)\log^{-2} x\end{equation}
and let 
\begin{equation}T(x; D) = 
\sum_{\ell\le D}\left\lvert \underset{xD^{-1} < z\le x^2D^{-2}}{\sum_{\ell m \le x}}a_{\ell m}\mu(m)\right\rvert\end{equation}

Then, we have that
\begin{equation}
\sum_{n\le  x} a_n\Lambda(n) = HA(x)\left\{1 + O((\log x)^{-1})\right\} + O(T(x, D)\log x)
\end{equation}
where $A(x) = A_1(x), g(d) = M_d(x)/A(x)$, and $$H = \prod_{p}(1 - g(p))\left(1 - \frac{1}{p}\right)^{-1}$$
\end{proposition}

\begin{proof}
This is Theorem 18.6 in \cite{FI2} for our particular sequence.
\end{proof}

\section{The remainder term}

In this section, we verify that (2.1) holds. From this point on $e(\alpha) = e^{2\pi i\alpha}$. 
First, we study the distribution of the roots of the congruence $v^2 + 3 \equiv 0\pmod{d}$ by studying Weyl sums related 
to these quadratic roots. In order to do so, we will establish a well-spacing of the points $v/d\pmod{1}$.
It is easy to show that for odd $d$, the roots to $v^2 + 3\equiv 0\pmod{d}$ each correspond to a representation
$$d = r^2 + rs + s^2 = \frac{(r - s)^2 + 3(r + s)^2}{4}$$ subject to
$(r, s) = 1, -r - s < r - s \le r + s$ where $v(r - s) \equiv (r + s)\pmod{d}$. 

It then follows that $$\frac{v}{d} \equiv -\frac{4(\overline{r - s})}{r + s} + \frac{r - s}{d(r + s)}\pmod{1}$$
where $\overline{r - s}$ is such that $(r - s)(\overline{r - s}) \equiv  1\pmod{r + s}$.

Note that we then have that $$\frac{|r - s|}{d(r + s)} < \frac{1}{2(r + s)^2}$$

Now, restrict $d$ to the range $4D < d \le 9D$. It then follows that $2D^{1/2} < r + s < 3D^{1/2}$, so for any two points 
$v_1/d_1, v_2/d_2$, $\max\left\{\frac{r_1 + s_1}{r_2 + s_2}, \frac{r_2 + s_2}{r_1 + s_1}\right\}\le \frac{3}{2}$
$$\left\lVert\frac{v_1}{d_1} - \frac{v_2}{d_2}\right\lVert > \frac{4}{(r_1 + s_1)(r_2 + s_2)} -
\max\left\{\frac{1}{(r_1 + s_1)^2}, \frac{1}{(r_2 + s_2)^2}\right\} \gg \frac{1}{D}$$

%Let $$\rho_h(d) = \sum_{v^2 + 3\equiv 0\pmod{d}} e\left(\frac{vh}{d}\right)$$. 

Then by the large sieve inequality of Davenport and Halberstam, we have the following
\begin{lemma} 
For all $\alpha_1, \alpha_2, \dots \in\CC$, we have that
$$
\underset{d\equiv 1\pmod{2}}{\sum_{D <  d\le 2D}}\sum_{v^2 + 3\equiv 0\pmod{d}} \left\lvert\sum_{n\le N} 
\alpha_n e\left(\frac{vn}{d}\right)\right\rvert^2\ll (D + N) \left(\sum_{n}\alpha_n^2\right).
$$
\end{lemma}

Applying Cauchy's inequality yields

\begin{proposition}
For all $\alpha_1, \alpha_2, \dots \in\CC$, we have that
\begin{equation}
\underset{d\equiv 1\pmod{2}}{\sum_{D <  d\le 2D}}\sum_{v^2 + 3\equiv 0\pmod{d}} \left\lvert\sum_{n\le N} 
\alpha_n e\left(\frac{vn}{d}\right)\right\rvert\ll D^{1/2}(D + N)^{1/2} \left(\sum_{n}\alpha_n^2\right)^{1/2}.
\end{equation}
\end{proposition}

Now, let $$\rho_h(d) = \sum_{v^2 + 3\equiv 0\pmod{d}} e\left(\frac{vh}{d}\right).$$

Then, the following holds: 

\begin{proposition}
\begin{equation}
\sum_{d\le D}\left\lvert\sum_{h\le N}\alpha_h\rho_h{d}\right\rvert\ll D^{1/2}(D + N)^{1/2}\left(\sum_{n}\alpha_n^2\right)^{1/2}.
\end{equation}
\end{proposition}

Now, we prove that (2.1) holds by proving the following: 

\begin{proposition}
For all $D\le x$
\begin{equation}
\sum_{d\le D} |r_d(x)| \ll D^{1/4} x^{3/4 + \epsilon}.
\end{equation}

\end{proposition}

\begin{proof}

Note that $$A_d(x) = \underset{\frac{r^2 + 3s^2}{4} \equiv 0\pmod{d}}{\sum_{\frac{r^2 + 3s^2}{4} \le x}} \gamma_r.$$

It is more convenient for now to consider only the contribution of the terms with $(r, d) = 1$. 
To that end, note that it is possible to replace $A_d(x)$ with
$$A_d^*(x) = \us{(r, d) = 1}{\us{\frac{r^2 + 3s^2}{4}\equiv 0\pmod{d}}{\sum_{\frac{r^2 + 3s^2}{4} \le x}}}\gamma_r$$
since 
$$\sum_{d\le D} |A_d(x) - A_d^*(x)| \le 
\sum_{d\le D} \sum_{\ell|d} |\gamma_\ell|\us{r^2 + 3s^2\equiv 0\pmod{4d}}{\sum_{r^2 + 3s^2\le 4x}} 1$$
$$ \le \sum_{\ell} |\gamma_\ell|\us{r^2 + 3s^2\equiv 0\pmod{4d}}{\sum_{r^2 + 3\le 4xs^{-2}}} \tau(r^2 + 3) \ll x^{1/2 + \epsilon}.$$

Now, rather than approximating $A_d^*(x)$, we shall approximate 
$$A_d^*(f) = \us{(r, d) = 1}{\sum_{r^2 + 3s^2 \equiv 0\pmod{4d}}} \gamma_rf\left(\frac{r^2 + 3s^2}{4}\right)$$ 

for some smooth $f$ supported on $[1, x]$ satisfying $$ f(u) = 1, \text{ for } y\le u \le x - y$$ $$f^{(j)}(x)\ll x^{-j}$$ 
where $y = \min\{x^{3/4}D^{1/4}, \frac{1}{2}x\}$. Note that bounding this is sufficient, since 
$$\sum_{d\le D}|A_d^*(f) - A_d^*(x)| \le \sum_{\ell^2 - \ell m + m^2\in I} \tau(\ell^2 - \ell m + m^2)\ll yx^{\epsilon}$$
where $I = \ZZ\cap([1, y]\cup[x - y, x])$. Note that since $\gamma_r$ is supported on odd primes, we have that 
$$A_d^*(f) = \sum_{v^2 + 3\equiv 0\pmod{4d}}\sum_{(r, d) = 1}\gamma_r\sum_{s\equiv vr\pmod{4d}} f\left(\frac{r^2 + 3s^2}{4}\right).$$
Now, let $$A_d(f) = \sum_{v^2 + 3\equiv 0\pmod{4d}}\sum_{r}\gamma_r\sum_{s\equiv vr\pmod{4d}} f\left(\frac{r^2 + 3s^2}{4}\right).$$
We can replace $A_d^*(f)$ with $A_d(f)$ with an error of $O(y\log x)$, which is small enough. 
We then have that by Poisson's formula $$A_d(f)= \frac{1}{4d}\sum_{r}\gamma_r\sum_{k\in\ZZ}\rho_{kr}(4d)F_r\left(\frac{k}{4d}\right)$$
where $$F_r(v) = \int_\RR f\left(\frac{r^2 + 3t^2}{4}\right) e(-vt) dt = 
2\int_0^{\infty} f\left(\frac{r^2 + 3t^2}{4}\right) \cos(2\pi vt) dt.$$

Note that the the contribution from when $k = 0$ is equal to $M_d(x) + O(y)$, so it is necessary and sufficient to bound the 
contribution from $k\ne 0$. To that end, note that by the change of variable $t = w\sqrt{x}/k,$ 
\begin{equation}
F_r\left(\frac{k}{4d}\right) = 
\frac{2\sqrt{x}}{k}\int_0^{\infty}f\left(\frac{r^2 + \frac{3xw^2}{k^2}}{4}\right)\cos\left(\frac{2\pi w\sqrt{x}}{4d}\right) dw.
\end{equation}
Integrating by parts twice yields that this equals
\begin{equation}
\frac{16\sqrt{x} d^2}{\pi^2k^3}\int_0^\infty\left(f^{'} + \frac{2w^2x}{k^2}f^{''}\right)\left(\frac{r^2 + \frac{3xw^2}{k^2}}{4}\right)
\cos\left(\frac{\pi w\sqrt{x}}{2d}\right) dw.
\end{equation}

Now, let 
$$R(f, D) = \sum_{D < d \le 2D} 
\left\lvert\frac{1}{4d}\sum_{r}\gamma_r\sum_{k\in\ZZ\setminus\{0\}}\rho_{kr}(4d)F_r\left(\frac{k}{4d}\right)\right\rvert.$$
We then have that 
$$R(f, D)\ll \frac{1}{D}\sum_{D < d \le 2D}\left\lvert\us{kr \ne 0}{\sum\sum}\gamma_rF_r\left(\frac{k}{4d}\right)\right\rvert.$$

To estimate this, we split this into sums with $|k|$ restricted to certain ranges. In particular, we write 
$$R_k(f, D) = \frac{1}{D}\sum_{D < d \le 2D} 
\left\lvert\sum_{2^k \le |k| < 2^{k + 1}}\sum_{r}\gamma_rF_r\left(\frac{k}{4d}\right) \right\rvert.$$

Then, we have that by $(3.4)$ and Proposition 3, $R_n(f, D)$ is

$$\frac{1}{D}\sum_{D < d \le 2D} 
\left\lvert\sum_{2^n \le |k| < 2^{n + 1}}\sum_{r}\gamma_r\rho_{kr}(d)\frac{2\sqrt{x}}{k}
\int_0^{\infty}f\left(\frac{r^2 + \frac{3xw^2}{k^2}}{4}\right)\cos\left(\frac{\pi w\sqrt{x}}{2d}\right) dw\right\rvert$$
$$\ll \frac{\sqrt{x}}{D}\int_0^{2^{n + 1}} \sum_{D < d \le 2D}
\left\lvert\sum_{2^n \le |k| < 2^{n + 1}}\sum_{r}\gamma_r\rho_{kr}(d)f\left(\frac{r^2 + \frac{3xw^2}{k^2}}{4}\right)\right\rvert dw$$
$$\ll \frac{x^{1/2 + \epsilon}}{D}D^{1/2}(D + 2^n\sqrt{x})^{1/2}(2^n\sqrt{x})^{1/2}.$$

Similarly, we also have that by $(3.5)$ and Proposition 3
$R_n(f, D)$ is  $$\ll \frac{D\sqrt{x}}{2^{3n}}\int_{0}^{2^{n + 1}}\sum_{D < d\le 2D}\left\lvert\sum_{2^n\le |k| < 2^{n + 1}} 
\sum_{r}\gamma_{r}\rho_{kr}(d)\left(f^{'}+\frac{2w^2x}{k^2}f^{''}\right)\left(\frac{r^2 + \frac{3xw^2}{k^2}}{4}\right)\right\rvert dw$$
$$\ll \frac{x^{3/2 + \epsilon}D^{3/2}}{y^22^{2n}}(D + 2^n\sqrt{x})^{1/2}(2^n\sqrt{x})^{1/2}$$

by Proposition 2.

Proposition 4 then follows from summing over all $n$.

\end{proof}

\section{The bilinear form}

Now, we shall bound the bilinear form in (2.2) by estimating the following sum:

\begin{equation} 
B_1(M, N) = \sum_{N\le n\le N'}\left\lvert\sum_{M < m\le M'}a_{mn}\mu(m)\right\rvert
\end{equation}

for some unspecified $M < M'\le 2M, N < N'\le 2N$ by showing the following: 

\begin{proposition}

For $\delta$ a sufficiently small positive number, we have that 
\begin{equation}
B(M, N) \ll MN(\log MN)^{-A}
\end{equation}

for all $A > 0$, where M = $N^{\delta}$.
\end{proposition}

\begin{proof}

First, note that it is sufficient to estimate 
\begin{equation}
B_1(M, N) = \sum_{N < n\le N'}\left\lvert\us{(m, n) = 1}{\sum_{M < m\le M'}}a_{mn}\mu(m)\right\rvert
\end{equation}

since if $(m, n) = d$, if $d < M^{1/2}$, we can just transfer the factor of $d$ to $n$, and otherwise use the trivial bound.

Write $\gamma(\mathfrak{a})$ to denote $\gamma_{2\Re\mathfrak{a}}$.

Note that we have that $$a_n = \sum_{N\mathfrak a = n}\gamma(\mathfrak a)$$ so by unique factorization in $\QQ(\omega)$, we have that
for relatively prime $m, n$, we have that $$a_{mn} = \frac{1}{6}\sum_{N\mathfrak m = m}\sum_{N\mathfrak n = n} \gamma(\mathfrak{mn})$$
where the factor of $1/6$ accounts for the six units $\pm 1, \pm\omega, \pm\omega^2$ in $\ZZ[\omega]$. It follows that 
$$B_1(M, N) = 
\frac{1}{6}\sum_{N < N(\mathfrak{n})\le N'}\left\lvert
\us{(\mathfrak{m, n}) = 1}{\sum_{M < N(\mathfrak{m})\le M'}}\gamma(\mathfrak{mn})\mu(\mathfrak{m})\right\rvert.$$

The coprimality condition can easily be dropped by a similar argument by which it was added, so it follows that it is sufficient to
show that $$B_2(M, N) = \sum_{N < N(\mathfrak{n})\le N'}\left\lvert
\sum_{M < N(\mathfrak{m})\le M'}\gamma(\mathfrak{mn})\mu(\mathfrak{m})\right\rvert \ll MN(\log MN)^{-A}$$
By Cauchy, we have that it is sufficient to show that $$B_3(M, N) = \sum_{N < N(\mathfrak{n})\le N'}\left\lvert
\sum_{M < N(\mathfrak{m})\le M'}\gamma(\mathfrak{mn})\mu(\mathfrak{m})\right\rvert^2 \ll M^2N(\log MN)^{-A}.$$

We then have that

$$B_3(M, N) = 
\sum_{M < N(\mathfrak{m}_1), N(\mathfrak{m}_2)\le M'} \mu(\mathfrak{m}_1)\mu(\mathfrak{m}_2)S(\mathfrak{m}_1, \mathfrak{m}_2)$$

where $$S(\mathfrak{m}_1, \mathfrak{m}_2) = \sum_{N < N(\mathfrak{n}) \le N'} \gamma(\mathfrak{nm}_1)\gamma(\mathfrak{nm}_2).$$

Now, let $\ell_1, \ell_2$ be such that
$$\mathfrak{nm}_1 + \conj{\mathfrak{n}}\conj{\mathfrak{m}}_1 = \ell_1$$
$$\mathfrak{nm}_2 + \conj{\mathfrak{n}}\conj{\mathfrak{m}}_2 = \ell_2$$

and let $\Delta(\mf m_1, \mf m_2) = \Delta = i(\mf m_1\conj{\mf m}_2 - \conj{\mf m}_1\mf m_2)$. 
Note that $\ell_1, \ell_2\le 4\sqrt{MN}$. 
When $\Delta = 0$, note that the contribution $B_0(M, N)$ satisfies
$$B_0(M, N) \ll N(\log N)^2\us{\Im \conj{\mf m}_1\mf m_2 = 0}{\sum\sum} 1$$
which is clearly $\ll NM^2(\log MN)^{-A}$. 

Otherwise, we have that $$\conj{\mf{n}} = \frac{i(\ell_1\mf m_2 - \ell_2\mf m_1)}{\Delta}$$
so it follows that $$\ell_1\mf m_2 \equiv \ell_2\mf m_1 \pmod{\Delta}$$ and that 
$$\Delta^2N < N(\ell_1\mf m_2 - \ell_2\mf m_1) \le \Delta^2N'.$$

It then follows that 
$$S(\mf m_1, \mf m_2) = 
\us{\Delta^2N < N(\ell_1\mf m_2 - \ell_2\mf m_1)\le \Delta^2N'}{\sum_{\ell_1\mf m_2\equiv \ell_2\mf m_1\pmod{\Delta}}}.
\gamma_{\ell_1}\gamma_{\ell_2}$$

Now, we state Proposition 20.9 in \cite{FI1}, which is used below: 
\begin{proposition}
$$\sum_{q\le Q}\us{y\in\RR}{\us{\mathfrak{a}\in\CC}{\us{a\in\ZZ, (a, q) = 1}{\max}}}
\left\lvert\us{\ell_1\equiv a\ell_2\pmod{q}}{\us{|\ell_1 - \mathfrak{a}\ell_2|\le y}{\us{\ell_1, \ell_2\le x}{\sum\sum}}}
\gamma_{\ell_1}\gamma_{\ell_2} 
- \phi(q)^{-1}\us{|\ell_1 - \mathfrak{a}\ell_2|\le y}{\us{\ell_1, \ell_2\le x}{\sum\sum}}\gamma_{\ell_1}\gamma_{\ell_2}\right\rvert
\ll x^2(\log x)^{-A}$$
where $Q = x(\log x)^{-B}$ for some $B > 0$ that depends on $A$.
\end{proposition}

We can split up $S(\mf m_1, \mf m_2)$ into classes restricted to $$\ell_1 \equiv a\ell_2\pmod{\Delta}$$ 
for $a\in(\ZZ/(\Delta))^*$ such that $a\mf m_2 \equiv \mf m_1\pmod{\Delta}$ and apply Proposition 6.
It then follows that $$B_0(M, N) \ll B_4(M, N) + O(NM^2(\log MN)^{-A})$$
where $$B_4(M, N) = \us{M < N(\mf m_1), N(\mf m_2) \le M'}{\sum\sum}\mu(\mf m_1)\mu(\mf m_2)\frac{\eta(\Delta)}{\phi(\Delta)}
\us{\Delta^2 N < N(\ell_1\mf m_2 - \ell_2\mf m_1)\le \Delta^2N'}{\us{\ell_1, \ell_2\le x}{\sum\sum}}\gamma_{\ell_1}\gamma_{\ell_2}$$
where $\eta(\Delta)$ is the total number of $a\in\left(\ZZ/(\Delta)\right)^*$ such that $a\mf m_2\equiv \mf m_1\pmod{\Delta}$.

By the prime number theorem, we have that the inner sum satisfies 
$$\us{\Delta^2 N < N(\ell_1\mf m_2 - \ell_2\mf m_1)\le \Delta^2N'}{\us{\ell_1, \ell_2\le x}{\sum\sum}}\gamma_{\ell_1}\gamma_{\ell_2}
= X + O(MN(\log MN)^{-A})$$
where $$X = \us{\Delta\sqrt{N} < |\ell_1\mf m_2 - \ell_2\mf m_1|\le \Delta\sqrt{N'}}{\int\int}d\ell_1d\ell_2 = 
|\Delta|\us{N < |u + \omega v|\le N'}{\int\int}du dv = \frac{1}{2}\pi\sqrt{3}|\Delta|(N' - N).$$

It therefore now remains to estimate 
$$S_1 = \us{M < N(\mf m_1), N(\mf m_2) \le M'}{\sum\sum}\mu(\mf m_1)\mu(\mf m_2)\frac{\eta(\Delta)|\Delta|}{\phi(\Delta)}.$$
Splitting this up for all $(\mf m_1, \mf m_2) = \mf d$, we then have that 
$$S_1 = \sum_{\mf d} \mu^2(d)\us{(\mf m_1, \mf m_2) = (\mf m_1\mf m_2) = 1}{\us{M < N(\mf m_1\mf d), N(\mf m_2\mf d)\le M'}{\sum\sum}}
\mu(\mf m_1\mf d)\mu(\mf m_2\mf d)\frac{\eta(\Delta N(\mf d))|\Delta|N(\mf d)}{\phi(\Delta N(\mf d))} $$
$$= \sum_{\mf d} \mu^2(d)\us{(\mf m_1, \mf m_2) = (\mf m_1\mf m_2) = 1}{\us{M < N(\mf m_1\mf d), N(\mf m_2\mf d)\le M'}{\sum\sum}}
\mu(\mf m_1)\mu(\mf m_2)\frac{\eta(\Delta N(\mf d))|\Delta|N(\mf d)}{\phi(\Delta N(\mf d))}.$$

Note that we have that 
$$\eta(\Delta N(\mf d)) = \us{a\equiv \mf m_2\mf m_1^{-1}\pmod{\conj{\mf d}\Delta}}{\sum_{a\in(\ZZ/(\Delta N(\mf d)))^*}} 1
= N(\mf d)\prod_{p|N(\mf d), p\nmid \Delta} \left(1 - \frac{1}{p}\right).$$

It then follows that 
$$S_1 = 
\sum_{\mf d} \mu^2(d)N(\mf d)\us{(\mf m_1, \mf m_2) = (\mf m_1\mf m_2) = 1}{\us{M < N(\mf m_1\mf d), N(\mf m_2\mf d)\le M'}{\sum\sum}}
\mu(\mf m_1)\mu(\mf m_2)\frac{|\Delta|}{\phi(\Delta)}.$$

By multiplicativity, we have that $$\frac{|\Delta|}{\phi(\Delta)} = \sum_{d | \Delta}\mu^2(d)\phi(d)^{-1}.$$

Using this and reversing the order of summation, we have that 
$$S_1 = 
\sum_{\mf d} \mu^2(d)N(\mf d)\us{(\mf m_1, \mf m_2) = (\mf m_1\mf m_2) = 1}{\us{M < N(\mf m_1\mf d), N(\mf m_2\mf d)\le M'}{\sum\sum}}
\mu(\mf m_1)\mu(\mf m_2)\sum_{d|\Delta}\mu^2(d)\phi(d)^{-1}
$$
$$ = 
\sum_{\mf d} \mu^2(d)N(\mf d) \sum_{d\le 2M}\phi(d)^{-1}
\us{\mf m_1\conj{\mf m}_2\equiv \conj{\mf m}_1\mf m_2\pmod{d}}
{\us{(\mf m_1, \mf m_2) = (\mf m_1\mf m_2) = 1}{\us{M < N(\mf m_1\mf d), N(\mf m_2\mf d)\le M'}{\sum\sum}}}\mu(\mf m_1)\mu(\mf m_2)$$
$$
= 
\sum_{\mf d} \mu^2(d)N(\mf d) \sum_{d\le 2M}\phi(d)^{-1}
\frac{1}{d}\sum_{\chi}
{\us{(\mf m_1, \mf m_2) = (\mf m_1\mf m_2) = 1}{\us{M < N(\mf m_1\mf d), N(\mf m_2\mf d)\le M'}{\sum\sum}}}\mu(\mf m_1)\mu(\mf m_2)
\psi(\mf m_1)\conj{\psi}{(\mf m_2)}.
$$

by orthogonality where $\chi$ runs over the characters of $\ZZ[\omega]/(d)$ and $\psi(\mf m) = \chi(\mf m)\conj{\chi}(\conj{\mf m}).$

To estimate this, we use the following version of the Siegel-Walfisz Theorem that follows from the main result in \cite{Gol}:

\begin{proposition}
For any character $\psi$ on ideals
$$\sum_{N(\mf m) \le x}\mu(\mf m)\psi(\mf m)\ll_A x(\log x)^{-A}$$
for all $A > 0$.
\end{proposition}

Now, let $$S_{\mf d, d, \psi}^*(M) = 
{\us{(\mf m_1, \mf m_2) = (\mf m_1\mf m_2, \mf d) = 1}{\us{M < N(\mf m_1\mf d), N(\mf m_2\mf d)\le M'}{\sum\sum}}}
\mu(\mf m_1)\mu(\mf m_2)\psi(\mf m_1)\conj{\psi}{(\mf m_2)}.$$

Then, it is easy to see that $$S_{\mf d, d, \psi}^*(M) = S_{\mf d, d, \psi}(M) + O(M^{1 + \epsilon})$$
where $$S_{\mf d, d, \psi}(M) = 
{\us{(\mf m_1, \mf m_2) = 1}{\us{M < N(\mf m_1\mf d), N(\mf m_2\mf d)\le M'}{\sum\sum}}}
\mu(\mf m_1)\mu(\mf m_2)\psi(\mf m_1)\conj{\psi}{(\mf m_2)}.$$

We then have that $$\sum_{\mf d_1\in\ZZ[\omega]\setminus\{0\}}\mu^{2}(\mf d_1)S_{\mf d, d, \psi}(M/N(\mf d_1))$$
$$= {\left({{\us{M < N(m_1\mf d)\le M'}{\sum}}}\mu(\mf m_1)\psi(\mf m_1)\right)}
\left({{\us{M < N(m_2\mf d)\le M'}{\sum}}}\mu(\mf m_2)\conj{\psi}(\mf m_2)\right)$$

so by a variant of M\"{o}bius inversion, we have that $$S_{\mf d, d, \psi}(M) \ll (M/N(\mf d))^2(\log M/N(\mf d))^{-A}.$$

The desired result then follows.
\end{proof}

\section{Acknowledgements}
The author is grateful to J. B. Friedlander for feedback regarding this paper. The author is especially grateful to D. Goldston for feedback and guidance on this paper.

\end{document}